\documentclass[10pt,twoside,a4paper,reqno]{amsart}
\usepackage{fix-cm}
\usepackage[T1]{fontenc}
\usepackage{amscd,epsfig}
\usepackage{graphics,graphicx}

\theoremstyle{definition}
\newtheorem{theo+}           {Theorem}
\newtheorem{prop+}           {Proposition}
\newtheorem{coro+}           {Corollary}
\newtheorem{lemm+}           {Lemma}
\newtheorem{conjecture}      {Conjecture}

\newtheorem{defi+}           {Definition}
\newtheorem{problem}         {Problem}

\newtheorem*{pb-prime}       {Problem 1$\mathbb{'}$}

\newtheorem{not+}            {Notation}

\theoremstyle{remark}
\newtheorem{rema+}           {Remark}
\newtheorem{example}         {Example}

\newenvironment{theorem}{\begin{theo+}}{\end{theo+}}
\newenvironment{proposition}{\begin{prop+}}{\end{prop+}}

\newenvironment{lemma}{\begin{lemm+}}{\end{lemm+}}
\newenvironment{remark}{\begin{rema+}}{\end{rema+}}
\newenvironment{definition}{\begin{defi+}}{\end{defi+}}

\newcommand{\al}{\alpha}
\newcommand{\be}{\beta}
\newcommand{\ga}{\gamma}

\newcommand{\bC}{\mathbb C}

\newcommand{\bR}{\mathbb R}

\newcommand{\cD}{\mathcal D}

\newcommand{\C}{\mathcal C}

\newcommand{\eps}{\epsilon}

\newcommand{\conv}{\mathrm{conv}}

\newcommand{\RR}{\mathcal R}

\newcommand{\De}{\Delta}

\newcommand{\Area}{\mathrm{Area}}
\newcommand{\Ar}[1]{|[{#1}]|}
\newcommand{\x}{\mathbf {x}}
\newcommand \M{\mathfrak{M}}
\newcommand \MM {\mathcal M}
\newcommand \prt {\partial}

\numberwithin{equation}{section}

\begin{document}

\title[Polygonal measures with vanishing harmonic moments ]
{On  polygonal measures with vanishing harmonic moments}

\author [Dmitrii V. Pasechnik] {Dmitrii V.  Pasechnik}
\address {School of Physical and Mathematical Sciences, Nanyang Technological University, 21 Nanyang Link, 637371 Singapore}
\email{dima@ntu.edu.sg}

\author[Boris Shapiro]{Boris Shapiro}
\address{Department of Mathematics, Stockholm University, SE-106 91 Stockholm,
      Sweden}
\email{shapiro@math.su.se}

\date{\today}

\keywords{potential theory, harmonic moments, polygonal measures}
\subjclass[2010]{Primary 44A60; Secondary 31B20}
\begin{abstract}  
A  polygonal measure is the sum of finitely many 
real constant density measures supported on triangles in $\bC$. 
Given a finite set $S\subset \bC$, we study the existence of  polygonal measures spanned by
triangles with vertices in $S$, which have all  harmonic
moments vanishing.  For $S$ generic, we show that the dimension of the linear space of
such measures is $\binom{|S|-3}{2}$. 

We also investigate the situation where the resulting density 
attains only values $0$ or $\pm 1$, which corresponds to pairs of polygons
of unit density having the same logarithmic potential at $\infty$.
We show that such a signed measure does not exist if $|S|\leq 5$, but  
for each $n\geq 6$ there exists an $S$, with $|S|=n$, giving rise to such a signed measure.
\end{abstract}

\dedicatory{To the memory of Andrei Zelevinsky}
\maketitle

\section{Introduction and main results}\label{s1}

Inverse problems in logarithmic potential theory have attracted substantial  attention since the publication of the fundamental paper \cite{Nov}, where P.S.~Novikov, in particular,  proved that two convex (or, more generally, star-shaped) domains in $\bC$  with unit  density  cannot have the same logarithmic potential near $\infty$. Notice that  the knowledge of the germ of a logarithmic potential of a finite compactly supported Borel  measure $\mu$  at $\infty$  is equivalent to the knowledge  of the sequence of its harmonic moments $m_j(\mu),\;j=0,1,\ldots ,$ where the $j$-th harmonic moment of  $\mu$  is defined by: 
$$m_j(\mu)=\int_{\bC}z^jd\mu(z).$$

More precisely, if $$\mathfrak u_\mu(z):=\int_\bC \ln|z-\xi|d\mu(\xi)$$ is the logarithmic potential of  $\mu$ and 
$$\mathfrak C_\mu(z):=\int_\bC \frac{d\mu(\xi)}{z-\xi}=\frac{\mathfrak \prt \mathfrak u_\mu(z)}{\prt z}$$ is its Cauchy transform then the Taylor expansion of $\mathfrak C_\mu(z)$  at $\infty$ has the form: 
$$\mathfrak C_\mu(z)=\frac{m_0(\mu)}{z}+\frac{m_1(\mu)}{z^2}+\frac{m_2(\mu)}{z^3}+\ldots .$$

Thus Novikov's result can be reformulated as the statement that two convex domains in $\bC$ with unit density cannot have coinciding sequences of harmonic moments. It is well-known that already for non-convex domains with unit density the
uniqueness in this problem no longer holds. For instance, examples of  pairs
of non-convex polygons with the same logarithmic potential near $\infty$ 
can be found on \cite[p.~333]{BroSt}, see Fig.~\ref{fg:TAD} below. 
The class of general polygons as well as domains bounded by lemniscates has
attracted a substantial attention in this area.     Several authors have also
considered  the class  of  polynomial densities instead of the unit density. 

By a {\it convex polygon} we mean  the convex hull of finite many points in the
plane, at least $3$ of which are non-collinear. A general {\it polygon} is  the
set-theoretic union of finitely many convex polygons.  By a {\it vertex} of a
polygon we mean a point of its boundary such that its sufficiently small
$\eps$-neighborhood in the polygon is different from  a half-disk of radius
$\eps$.   

Given an open set $D\subset \bC$, define its {\em standard measure } $$\mu_D=\chi_Ddxdy,$$ where $\chi_D$ is the characteristic function of $D$. 
The same measure is associated with the closure of $D$.
We say that two polygons in $\bC$ are {\em equipotential} if their standard measures create coinciding logarithmic potential outside their union.  Below we present one of the simplest examples of pairs of equipotential polygons  given in \cite[Example~1]{BroSt}. 

\begin{example}\label{ex:gabr}
Consider the $6$-tuples 
$T=\{\pm\sqrt{3}\pm I,\pm 2I\}$ and
$T'=\{\pm\frac{1\pm\sqrt{3}I}{2},\pm 1\}$. Let $F\subset\bC$ be 
the difference of the convex hull of $T$ and the union of the set of 6 triangles
obtained as the orbit of the triangle with nodes 
$(\sqrt{3}+ I,\sqrt{3}- I, 1)$ under the rotation by $\frac{\pi}{3}$, see Fig.~\ref{fg:TAD}.
Let $F'\subset\bC$ be the difference of the convex hulls of $T$ and of $T'$.
Then 
 $F$ and $F'$ have the same
logarithmic potential.  
\end{example}

\begin{figure}[h]
\begin{center}
\begin{picture}(360,100)(0,0)
\begingroup
\everymath{\scriptstyle}
\scriptsize

\put(77,-3){$-2I$}
\put(82,85){$2I$}
\put(35,14){$-\sqrt{3}-I$}
\put(135,14){$\sqrt{3}-I$}
\put(35,70){$-\sqrt{3}+I$}
\put(135,70){$\sqrt{3}+I$}
\put(105,55){$e^{\frac{\pi I}{3}}$}
\put(105,25){$e^{-\frac{\pi I}{3}}$}
\put(70,55){$e^{\frac{2\pi I}{3}}$}
\put(70,25){$e^{-\frac{2\pi I}{3}}$}
\put(68,40){$-1$}
\put(118,40){$1$}

\put(50,20){\line(2,3){15}}
\put(50,20){\line(1,0){30}}
\put(80,20){\line(2,-3){15}}
\put(95,-3){\line(2,3){15}}
\put(110,20){\line(1,0){30}}
\put(140,20){\line(-2,3){15}}
\put(50,65){\line(2,-3){15}}
\put(50,65){\line(1,0){30}}
\put(80,65){\line(2,3){15}}
\put(95,88){\line(2,-3){15}}
\put(110,65){\line(1,0){30}}
\put(140,65){\line(-2,-3){15}}

\put(180,20){\line(0,1){45}}
\put(180,20){\line(2,-1){45}}
\put(180,65){\line(2,1){45}}
\put(225,88){\line(2,-1){45}}
\put(270,65){\line(0,-1){45}}
\put(270,20){\line(-2,-1){45}}

\put(210,20){\line(1,0){30}}
\put(210,20){\line(-2,3){15}}
\put(195,43){\line(2,3){15}}
\put(210,65){\line(1,0){30}}
\put(240,65){\line(2,-3){15}}
\put(240,20){\line(2,3){15}}

\endgroup
\end{picture}
\end{center}
\caption{Two equipotential polygons: $F$ on the left, $F'$ on the right. \label{fg:TAD}}
\end{figure}
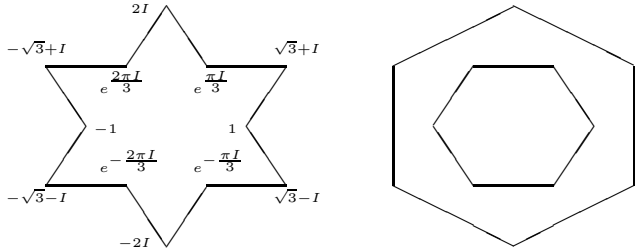

Notice that if different polygons with constant (but not necessarily unit)
density  have the same logarithmic potential near $\infty$ then they must have
the same set  of vertices, see \cite[Corollary~2 and Lemma~2]{BroSt}. (The
coincidence of the logarithmic potential near $\infty$ implies even more
restrictions on the polygons than just the coincidence of their set of
vertices, cf. \cite{BroSt}.)  

Taking this fact into account we pose the following  {\it classical inverse logarithmic potential problem for polygons in $\bC$}. 

\begin{problem}\label{prob:one} Given a finite set $S\subset \bC$, determine
whether there exist two equipotential polygons whose sets of vertices 
coincide with $S$.  
\end{problem}

One can  show that for generic $S$ no  pairs of equipotential polygons exist. 

\begin{definition} \label{def:sigmes}
A complex (respectively, real) {\em polygonal measure}  $\mu:=\mu(\cD)$ is the sum 
\begin{equation}\label{eq:noodes}
\mu:=\sum_{\Delta\in\cD} c_\Delta\mu_\Delta, \quad c_\Delta\in\bC\ 
\text{(respectively, $c_\Delta\in\bR$)},
\end{equation}
where $\cD$ is a finite set of closed triangles in the plane. 
The set of vertices of the triangles $\Delta\in\cD$ with $c_\Delta\neq 0$
in \eqref{eq:noodes}  is called the set of nodes of this decomposition.
\end{definition}

Notice that the decomposition \eqref{eq:noodes} of a given $\mu$ need not be
unique, and different choices of $\cD$ can lead to different sets of nodes. 

Besides the nodes of decompositions  \eqref{eq:noodes} of $\mu$ it is natural to talk about 
the {\em vertices} of $\mu$. They 
are $v\in\bC$ such that for any sufficiently small 
$\epsilon>0$ 
the restriction of the density of $\mu$ to the $\epsilon$-disk centered at 
$v$ is neither constant, nor there exists a line through $v$ dividing the disk into
two halves with different constant densities.
 
Obviously, the set of vertices of $\mu$ is a subset of the set of intersections of sides of the 
triangles in $\cD$. There exists a finite collection $\tilde\cD$ of triangles
with pairwise empty intersections of interiors, such that $\mu=\mu(\tilde\cD)$,
and nodes and vertices of $\mu$ coincide.
However, such a representation of $\mu$ need not be the most economic one, cf. 
e.g. Example~\ref{ex:gabr}.

Namely, in notation of  Example~\ref{ex:gabr}, consider
$\tilde\mu:=\mu_F-\mu_{F'}$.  Observe that
$\tilde\mu$ can be represented using only $6$ nodes, although the polygons
themselves have 12 vertices!  This also illustrates the non-uniqueness of
representation of $\tilde\mu$ in the form  \eqref{eq:noodes}. Indeed,
\[
2\tilde\mu=\sum_{0\leq j\leq 5} \mu_{\exp(\frac{j\pi I}{3})(\sqrt{3}+ I,\sqrt{3}- I, -2I)} - 
\sum_{0\leq j\leq 1}\mu_{(\sqrt{3}+(-1)^j I,\sqrt{3}-(-1)^j I,-\sqrt{3}+(-1)^j I)}.
\]



\medskip

Let an $S$ admit a pair of equipotential polygons. Taking the difference of their standard measures, one obtains a  polygonal  measure supported on the convex hull $\conv(S)$ of $S$ 
with density attaining only values $0,\pm 1$  and with all harmonic moments vanishing. 
Conversely, if one can find a polygonal  measure with all vanishing harmonic moments, 
and such that its density attains only values $0,\pm 1$, then one obtains a pair of equipotential  polygons
by taking the differences of $\conv(S)$ with 
the sets where the density attains value  $1$,  respectively $-1$.

If we weaken the condition that the density of a polygonal measure
attains only values $0, \pm 1$ then we arrive at the  setup of the present  paper. 
Given a {\em spanning} set $S$ (i.e. $S$  contains at least $3$ non-collinear
points), we introduce the linear spaces $\M^\bR(S)\subset \M^\bC(S)$ of
real-valued, respectively, 
complex-valued polygonal measures obtained as real, respectively, 
complex  linear spans of the standard measures of all triangles with
vertices in $S$. Obviously, $\M^\bC(S)=\bC\otimes \M^\bR(S)$.
 
We make a further step in the study of (non-)uniqueness in
logarithmic potential theory by considering the following question.  

\begin{problem}
Given a finite set $S\subset \bC$,   determine the linear subspace $\M^\bR_{null}(S)\subset  \M^\bR(S)$ of real-valued  polygonal measures (resp. of complex-valued  polygonal measures
$\M^\bC_{null}(S)\subset  \M^\bC(S)$) with all harmonic moments vanishing. 
\end{problem} 

The main technical tool we use is the {\it normalized  generating function $\Psi_\mu(u)$ for harmonic moments} of a measure $\mu$, defined by 
\begin{equation}\label{eq:Psi}
\Psi_\mu(u)=\sum_{j=0}^\infty \binom {j+2}{2} m_j(\mu)u^j.
\end{equation}
 Notice that $\Psi_\mu(u)$ is closely related to the Cauchy transform $\mathfrak C_\mu(z)$ at $\infty$. Namely, 
$$\Psi_\mu(u)=\frac{1}{2}\frac{d^2}{du^2}\left(\sum_{j=0}^\infty m_j(\mu) u^{j+2}\right).$$
 At the same time for a compactly supported measure $\mu$  and  sufficiently large  $|z|$, $z\mathfrak C_\mu(z)=\sum_{j=0}^\infty m_j(\mu)/ z^{j}$. Thus for $|u|$ sufficiently small, 
$$\Psi_\mu(u)=\frac{1}{2}\frac{d^2}{du^2}\left(u\mathfrak C_\mu\left(\frac{1}{u}\right)\right).$$
Similar multivariate generating  functions were recently considered in \cite{GPSS}. 
Important in our consideration are the following observations.

\begin{proposition}\label{pr:triangle} 
For measures $\mu$ with compact support,   
\begin{equation}\label{eq:PsiInt}
\Psi_\mu(u)=\sum_{j=0}^\infty \binom {j+2}{2} m_j(\mu)u^j=\int\frac{d\mu(z)}{(1-uz)^3}.
\end{equation}
The normalized generating function $\Psi_\Delta(u)$  
of (the standard measure of) the triangle $\Delta\subset \bC$ whose  vertices are located at $a,b,c$ is given by 
$$\Psi_\Delta(u)=\frac{\Area\, \Delta}{(1-au)(1-bu)(1-cu)}.$$
\end{proposition}

Note that the integral transform in \eqref{eq:PsiInt} appears to be 
a variant of {\em Fantappi\`{e} transformation}, cf. \cite{APS}. 

\begin{definition}
We say that a finite set $S=\{z_0,z_1,\ldots, z_n\}$ of points in $\bC$ is {\em non-degenerate} if no three of its points are collinear.
\end{definition}

\begin{proposition}\label{pr:basis} For any non-degenerate set $S=\{z_0,z_1,\ldots, z_n\},\, n\ge 2$ of points in $\bC$ and any fixed non-negative integer $j\le n$, the set of (standard measures of) all triangles  with a node at  $z_j$ is a basis of the spaces $\M^\bR(S)$ and $\M^\bC(S)$. In particular,  
$$\dim_\bR\M^\bR(S)=\dim_\bC\M^\bC(S)=\binom{n}{2}.$$  
\end{proposition}

\medskip
We are interested in   linear subspaces $\M_{null}^\bR(S)\subset\M^\bR(S)$ (resp. $\M_{null}^\bC(S)\subset\M^\bC(S)$) of real-valued (resp. complex-valued)   measures having all  vanishing harmonic moments. 

\medskip
The main results of this paper are as follows. 

\begin{proposition}\label{pr:comp} For any non-degenerate set $S=\{z_0,z_1,\ldots, z_n\},\, n\ge 2$ of points in $\bC$, 
$$\dim_\bC\M_{null}^\bC=\binom{n-1}{2}.$$
\end{proposition}

\begin{example} For $n=3$ the space $\M_{null}^\bC(S)$ is spanned  by the complex-valued measure $\tilde \mu$  whose densities with respect to   the basis of  triangles $\Delta_{012},\Delta_{013},$  $\Delta_{023}$ are given by:  
$$\begin{cases}
d_{012}=(z_1-z_2)/|[012]|\\
d_{013}=(z_3-z_1)/|[013]|\\
d_{023}=(z_2-z_3)/|[023]|\\
\end{cases},$$
where $[i,j,k]=\det\begin{pmatrix}1&1&1\\x_i&x_j&x_k\\y_i&y_j&y_k\end{pmatrix}$ stands for twice the signed area of the triangle with nodes $z_i,z_j,z_k$  and $z_j=x_j+y_jI,$  $I$ being the imaginary unit. 
\end{example}

\begin{remark} 
For $S$ non-degenerate, the space $\M_{null}^\bC(S)$ projects isomorphically on the linear subspace of $\M^\bC(S)$ spanned by all triangles 
$\De_{0,i,j}$ where $2\le i<j\le n$. In other words, assigning arbitrarily complex-valued densities $d_{0,i,j},\; 2\le i <j\le n$ we can uniquely determine the densities $d_{0,1,j},\; j=2,\ldots , n$ to get a measure belonging to  $\M_{null}^\bC(S)$. 
\end{remark}

\begin{theorem}\label{th:more} For any non-degenerate  set $S=\{z_0,z_1,\ldots, z_n\},\, n\ge 2$ of points in $\bC$, 
 $$\dim_\bR \M_{null}^\bR(S)=\binom{n-2}{2}.$$

\end{theorem} 

\begin{remark} 
For $S$ non-degenerate, the space $\M_{null}^\bR(S)$ projects isomorphically on the linear subspace of $\M^\bR(S)$ spanned by all triangles 
$\De_{0,i,j}$ where $3\le i<j\le n$. In other words, arbitrarily real-valued densities $d_{0,i,j},\; 3\le i <j\le n$, uniquely determine the densities $d_{0,1,j},\; j=2,\ldots, n$ and $d_{0,2,j},\;j=3,\ldots, n$ of a measure belonging to  $\M_{null}^\bR(S)$. 
\end{remark}

\begin{theorem}\label{th:5} For any non-degenerate $5$-tuple $S=\{z_0,z_1,z_2,z_3,z_4\},$ the space $\M_{null}^\bR(S)$  is spanned by the real measure $\tilde \mu$  with densities with respect to   the basis of  triangles $\Delta_{012},\Delta_{013},$ $\Delta_{014},$ $\Delta_{023},\Delta_{024},\Delta_{034}$ given by:  
\begin{equation}\label{eq:dens}\begin{cases}
d_{012}=||z_1-z_2||^2[134][234]/|[012]|\\
d_{013}=||z_1-z_3||^2[124][234]/|[013]|\\
d_{014}=||z_1-z_4||^2[123][234]/|[014]|\\
d_{023}=-||z_2-z_3||^2[124][134]/|[023]|\\
d_{024}=-||z_2-z_4||^2[134][123]/|[024]|\\
d_{034}=-||z_3-z_4||^2[123][124]/|[034]|\\
\end{cases}
\end{equation}
\end{theorem}

\begin{example} 
For the $5$-tuple $\{0,2, 3+I, 1+3I, 2I\}$ the  
measure $3\tilde \mu$  is shown in Fig.~\ref{fg:T1} below.  (In this case $3\tilde\mu$ has integer densities which are easier to show \TeX nically.) 
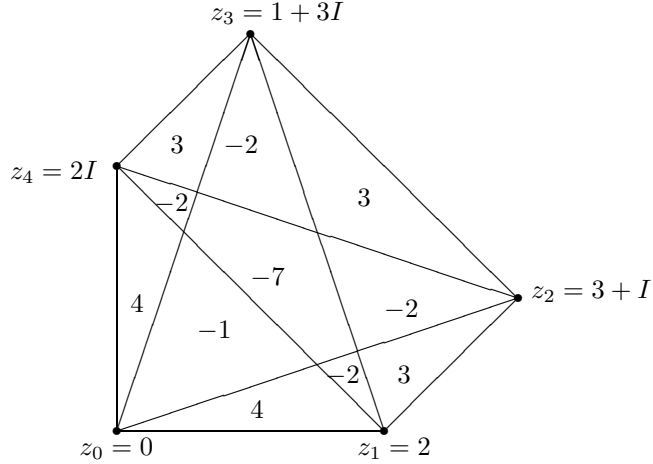
\begin{figure}[h]
\begin{center}
\begin{picture}(360,180)(0,0)
\put(120,20){\line(1,0){100}}
\put(220,20){\line(1,1){50}}
\put(120,20){\line(0,1){100}}
\put(120,120){\line(1,1){50}}
\put(170,170){\line(1,-1){100}}
\put(120,20){\line(3,1){150}}
\put(120,20){\line(1,3){50}}
\put(120,120){\line(1,-1){100}}
\put(120,120){\line(3,-1){150}}
\put(220,20){\line(-1,3){50}}

\put(106,11){$z_0=0$}
\put(210,11){$z_1=2$}
\put(275,70){$z_2=3+I$}
\put(155,175){$z_3=1+3I$}
\put(80,115){$z_4=2I$}

\put(170,25){$4$}
\put(198,38){$-2$}
\put(225,38){$3$}
\put(220,63){$-2$}
\put(170,75){$-7$}
\put(160,125){$-2$}
\put(210,105){$3$}
\put(150,55){$-1$}
\put(125,65){$4$}
\put(140,125){$3$}
\put(134,103){$-2$}

\put(120,20){\circle*{3}}
\put(220,20){\circle*{3}}
\put(270,70){\circle*{3}}
\put(170,170){\circle*{3}}
\put(120,120){\circle*{3}}

\end{picture}
\end{center}
\caption{Measure $3\tilde \mu$ spanning $\M_{null}^\bR(0,2,3+I,1+3I,2I)$.  \label{fg:T1}}
\end{figure}
\end{example}

\begin{remark}
Suppose that the densities  of a polygonal measure $\mu\in
\M^\bR_{null}(S)$  with respect to  the basic triangles containing a fixed node
(say $z_0$) are known. It is still desirable to find the densities in all its 
chambers, for instance in view of  the classical Problem~\ref{prob:one}.  
Here by a
{\it chamber} we mean a connected component of $\conv(S)\setminus Arr(S)$,
$Arr(S)$ being the union of all lines connecting pairs of points in $S$.
(Integers in Fig.~\ref{fg:T1} show the densities in the chambers they are
placed in.) Each chamber is contained in a number of basic triangles and the
density of a given chamber equals the sum of the densities of all basic
triangles containing it. Containment of chambers in  triangles (and more
generally in simplices in $\bR^d$) can be coded by an appropriate incidence
matrix whose rows correspond to  simplices  and columns correspond to
chambers. If a simplex contains a chamber then the corresponding entry 
of the incidence matrix equals
$1$, otherwise the entry equals $0$. 
Examples of incidence matrices are given in the proof of Theorem~\ref{th:5-tuple} below. 

This incidence matrix of chambers and simplices
in $\bR^d$ was for the first time studied in \cite {AlGelZel} and later in
\cite{Al1,Al2}.  It has rather delicate properties and already  the
number of chambers is a complicated function of the initial non-degenerate set
$S$. In particular, this number  can change  if we deform $S$ within the class
of non-degenerate sets. This observation partially explains why  results of
the present paper do not automatically solve Problem~\ref{prob:one}. 
\end{remark}

\begin{remark} Notice that if $S=\{z_0,\ldots, z_n\}$ consists of complex numbers having only  rational real and imaginary parts then one can choose a basis of $\M_{null}^\bR(S)$ consisting of polygonal measures with integer densities. 
\end{remark}

Using Example~\ref{ex:gabr} together with  Theorem~\ref{th:5} we can prove the
following result related to the classical Problem~\ref{prob:one}. 

\begin{theorem}\label{th:5-tuple}
For each $n\geq 6$ there exists $S$, with $|S|=n$,  admitting
a pair of equipotential polygons. No such $S$ exists if $|S|\leq 5$.
\end{theorem}

The essential part of the proof of Theorem~\ref{th:5-tuple} is to deal with the case $|S|=5$. 

\medskip
Our final result concerns a natural cone spanned by the standard measures of triangles with nodes in $S$. Namely, for an arbitrary non-degenerate set $S=\{z_0,z_1,\ldots, z_n\}$ denote by $\mathfrak K(S)\subset \M^\bR(S)$  the $\binom n 2$-dimensional cone   obtained by taking non-negative linear combinations of the  standard measures of all triangles with nodes in $S$. (Recall that $ \M^\bR(S)$ is the linear span of these measures.) 

\begin{theorem}\label{th:cone}
 Extreme rays of  $\mathfrak K(S)$ are spanned by (the standard measures)  of triangles which do not contain any point of $S$ different from its own nodes.  In particular, if $S$ is a convex configuration, (i.e. each $z_j$ belongs to the convex hull of $S$) then every triangle with nodes in $S$ spans an extreme ray of  $\mathfrak K(S)$.
 \end{theorem} 
 
We finish the introduction with a conjectural   description of all faces of
$\mathfrak K(S)$. We say that a pair of  triangles with vertices in $S$  {\it
forms a flip} if they have a common side and their convex hull is a $4$-gon.
With any pair of  triangles forming a flip we associate their {\em flipped
pair} obtained by removing the opposite diagonal from their convex hull, see
Case a) Fig.~\ref{fg:T2} below. (On this figure the pairs of triangles $(\De_{013},
\De_{123})$  and $(\De_{012}, \De_{023})$ form a flip and each pair is the
flipped one to the other pair.)  
 
 \begin{conjecture}\label{th:conefaces}
 A collection $Col$ of triangles having no internal vertices  spans a face of $\mathfrak K(S)$ if and only if  for each pair of triangles  from $Col$ forming a flip its flipped pair of triangles is also contained in $Col$. 
  \end{conjecture}
  
  The necessity of the stated condition is quite obvious and its sufficiency might follow from the results of \cite{AlGelZel}.

\subsection*{Acknowledgements.}   
B.S. is grateful to the Division of Mathematical Sciences of Nanyang Technological
University  for  hospitality in April 2012 when this project was carried out.
D.V.P. is supported by Singapore MOE Tier 2 Grant MOE2011-T2-1-090 (ARC 19/11).
The authors thank Sinai Robins for helpful discussions.
The authors are grateful to the anonymous referee for extremely useful comments
on the initial version of this paper.

\section{Proofs}\label{s2}

\begin{proof}[Proof of Proposition~\ref{pr:triangle}]
First, we prove \eqref{eq:PsiInt}.  Indeed, 
\begin{multline*}
\int \frac{d\mu(z)}{(1-uz)^3}=\sum_{k\geq  0} u^k \int \binom{k+2}{2}  z^k\,d\mu(z)
=\sum_{k\geq  0} u^k \binom{k+2}{2} m_k(\mu) =\Psi_\mu(u),
\end{multline*}
as required.
By \cite[(1)]{Dav}, for any $f(z)$ analytic in the closure of $\De$, we have 
\[
\frac{1}{2\Area\, \De}\int_\De f''(z) dx dy=\sum_{k=1,\ j\neq i\in\{1,2,3\}\setminus\{k\}}^k
\frac{f(z_k)}{(z_k-z_i)(z_k-z_j)}.
\]
Applying the latter identity and \eqref{eq:PsiInt}
 to $f(z)=\frac{1}{2u^2}\frac{1}{1-uz}$, we get the claimed formula.
\end{proof}

To prove Proposition~\ref{pr:basis} we need to recall some basic notions. First we present a description of  all  linear dependences among  the standard measures of all triangles with vertices in a non-degenerate set $S$. Namely, any $4$-tuple of points (say, $\{z_0,z_1,z_2,z_3\}$) in $S$ has $4$ triangles with vertices at these points. To study linear dependences between these $4$ triangles, one has to distinguish between two cases.  Consider
the convex hull of $\{z_0,z_1,z_2,z_3\}$, which is either a quadrangle or a triangle, see Fig.~\ref{fg:T2}.   
Obviously, in  Case a) 
  we have (up to  permutation of the vertices) the equality
$\mu_{\Delta_{013}}+\mu_{\Delta_{123}}=\mu_{\Delta_{023}}+\mu_{\Delta_{012}}$.  
Analogously, in Case b) 
we have   (up to  permutation of the vertices)  the relation $\mu_{\Delta_{012}}=\mu_{\Delta_{013}}+\mu_{\Delta_{123}}+\mu_{\Delta_{023}}$. 

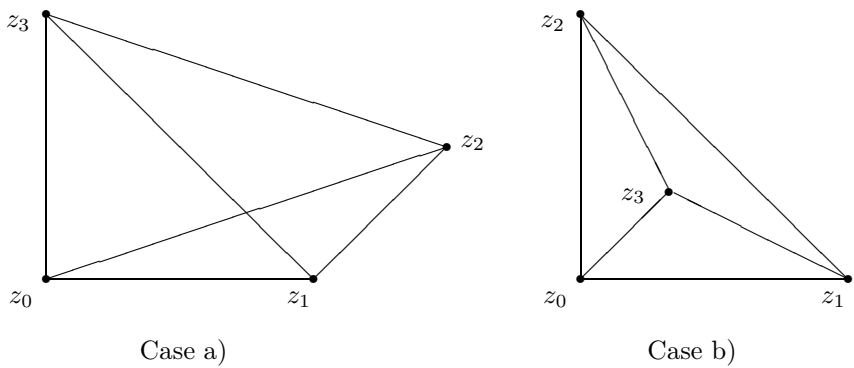
\begin{figure}[h]
\begin{center}
\begin{picture}(360,140)(0,0)
\put(40,40){\line(1,0){100}}
\put(140,40){\line(1,1){50}}
\put(40,40){\line(0,1){100}}
\put(40,40){\line(3,1){150}}
\put(40,140){\line(1,-1){100}}
\put(40,140){\line(3,-1){150}}

\put(26,31){$z_0$}
\put(130,31){$z_1$}
\put(195,90){$z_2$}
\put(25,135){$z_3$}

\put(75,10){Case a)}

\put(40,40){\circle*{3}}
\put(140,40){\circle*{3}}
\put(190,90){\circle*{3}}
\put(40,140){\circle*{3}}

\put(240,40){\line(1,0){100}}
\put(340,40){\line(-2,1){65}}
\put(240,40){\line(0,1){100}}
\put(240,40){\line(1,1){33}}
\put(240,140){\line(1,-1){100}}
\put(240,140){\line(1,-2){33}}

\put(226,31){$z_0$}
\put(330,31){$z_1$}
\put(255,70){$z_3$}
\put(225,135){$z_2$}

\put(240,40){\circle*{3}}
\put(340,40){\circle*{3}}
\put(273,73){\circle*{3}}
\put(240,140){\circle*{3}}

\put(265,10){Case b)}

\end{picture}
\end{center}
\caption{Linear dependence of $4$ triangles spanned by $4$ points.  \label{fg:T2}}
\end{figure}

To complete the proof of Proposition~\ref{pr:basis} we need to show that if $S$
is non-degenerate then the set of (the standard measures of) all triangles
containing a given vertex $z_j\in S$ spans $\M^\bR(S)$
and that this set is
linearly independent. The former immediately follows from the discussion preceding 
Fig.~\ref{fg:T2}. It remains to show the latter. We need more
notions.

\begin{definition} 
By a {\it $2$-chain} $\C^{(2)}$ we mean a formal linear combination 
\begin{equation}\label{2-chain}
\C^{(2)}=\al_1\De_1+\al_2\De_2+\ldots +\al_s\De_s
\end{equation}
 of triangles $\De_1,\ldots \De_s$ in $\bC$  with real or complex coefficients where each triangle is equipped with the standard orientation induced from $\bC$. 
\end{definition}

By using the standard pairing 
$$\langle fdxdy, \C^{(2)} \rangle =\int_{\C^{(2)}} fdxdy=\sum_{j=1}^s\al_j\int_{\De_j}fdxdy,$$
one sees that 
a $2$-chain \eqref{2-chain} defines a linear functional on the space $\Omega^{(2)}$ of smooth $2$-forms  on $\bC$. 

\begin{definition}
Analogously, by a {\it $1$-chain} $\C^{(1)}$ we mean a formal linear combination 
\begin{equation}\label{1-chain}
\C^{(1)}=\be_1I_1+\be_2I_2+\ldots +\be_tI_t
\end{equation}
 of oriented finite intervals  $I_1,\ldots I_s$ in $\bC$  with real or complex coefficients.  
\end{definition}

Again, by using the standard pairing 
$$\langle w, \C^{(1)} \rangle =\int_{\C^{(1)}} w=\sum_{j=1}^t\be_j\int_{I_j}w,$$
where $w$ is an arbitrary smooth $1$-form, 
one sees that 
a $1$-chain \eqref{1-chain} defines a linear functional on the space $\Omega^{(1)}$ of smooth $1$-forms  on $\bC$.  

\begin{definition}
For a given  triangle $\De$ with  vertices $a,b,c$ where  triple $(a,b,c)$ is
counterclockwise oriented  we  define its {\it boundary} $\prt \De$ as  the sum
of  three oriented intervals $[ab]+[bc]+[ca]$. As usual, we extend 
the boundary operator $\prt$ by linearity to the linear space of all $2$-chains.
\end{definition} 

\begin{definition} 
 A $2$-chain (resp. a $1$-chain) is called {\it vanishing} if it defines the zero linear functional on $\Omega^{(2)}$ (resp. $\Omega^{(1)}$).  
 \end{definition} 
 
\begin{lemma}\label{lm:depend}  A $2$-chain $\C^{(2)}$ is vanishing if and only if its boundary $\prt \C^{(2)}$ is a vanishing $1$-chain. 
\end{lemma} 

\begin{proof} Stokes theorem says that $\int_{\prt
\De}w=\int_\De d w $, where $w\in \Omega^{(1)}$, $\De$ is an arbitrary triangle,
$\prt \De$ is its boundary and $dw$ is the differential of $w$. (Recall that
if $w=F(x,y)dx+G(x,y)dy$ then $dw=(G'_x-F'_y)dxdy$.)  Observe that any $2$-form
$f(x,y)dxdy$ can be represented  as $dw_x$ where $w_x=F(x,y)dx$ and $F(x,y)$ is
the primitive function of $-f(x,y)$ along  vertical lines. Analogously,
$f(x,y)dxdy$ equals $dw_y$ where $w_y=G(x,y)dy$ and $G(x,y)$ is the primitive
function of $f(x,y)$ along  horizontal lines. Thus
$$\int_{\C^{(2)}}fdxdy=\int_{\prt{\C^{(2)}}} w_x=\int_{\prt{\C^{(2)}}}w_y.$$ If
the l.h.s. vanishes for all $fdxdy$ then ${\prt{\C^{(2)}}}$ should vanish and
vice versa.  \end{proof}

\begin{proof}[Proof of Proposition~\ref{pr:basis}] 
We need to show that for any non-degenerate $S$ the
standard measures of all triangles containing $z_0$ are linearly independent.
Indeed, by Lemma~\ref{lm:depend} a $2$-chain $\C^{(2)}$ of triangles vanishes
if and only its  boundary chain $\prt \C^{(2)}$ vanishes,  But if $S$ is
non-degenerate then each triangle $\De_{0,i,j}$ has its unique edge $(z_i,z_j)$ in
the boundary and no chain of the form $\be_{i,j}(z_i,z_j)$ with non-trivial
$\be_{i,j}$ can be vanishing.  Therefore the standard measures of triangles
$\De_{0,i,j}$  form a basis in $\M^\bC(S)$ and $\M^\bR(S)$.   
\end{proof}

\begin{remark}\label{cor:lindep}
Proposition~\ref{pr:basis} has an interesting and immediate corollary, 
that the linear dependences among the standard measures of all 
triangles with vertices in $S$ are generated by the linear dependences shown on Fig.~\ref{fg:T2} 
which come  from all possible $4$-tuples of vertices in $S$.

It is a special case of  \cite[Theorem~1]{AlGelZel}. 
(Unfortunately, it seems that a  proof of this important statement is missing in the available literature.) 
J.A.~De~Loera informed us that it can be derived from results in \cite{DLHSS}
(e.g. in the plane case one can use Lawson Theorem),  or \cite{DLRS}. 
\end{remark}

%

\begin{proof}[Proof of Proposition~\ref{pr:comp}] 
The case $n=2$ is trivial, so we assume $n\geq 3$.  
Given a non-degenerate $S=\{z_0, z_1,\ldots , z_n\}$,  consider the complex-valued  measure $\mu$ obtained by assigning (complex) densities $d_{0ij},\;  1\le i<j\le n$ to  triangles $\De_{0ij}$. Set $m_{i,j}=d_{0ij}\Area\,\Delta_{0ij}$. Then the normalized generating function $\Psi_\mu(u)$  for harmonic moments of $\mu$ is given by 
 \begin{equation}\label{eq:main2}
\begin{split}
 \Psi_\mu(u)=\sum_{1\le i <j\le n} d_{0ij}\Psi_{\De_{0ij}}(u)=
\sum_{1\le i <j\le n}\frac{m_{ij}}{(1-z_0u)(1-z_iu)(1-z_ju)}\\
=\frac{1}{1-z_0u}\frac{P(u)}{\prod_{i=1}^n(1-z_iu)},
\end{split}
 \end{equation} 
where $P(u)$ is a polynomial of degree at most $n-2$. Its coefficients at $1$,
$u$, $u^2$,\ldots, $u^{n-2}$ are the consecutive entries of the vector
$\MM_n^\bC\cdot \mathbf{m}_n$, where  $$\mathbf{m}_n=(m_{12},
m_{13},\ldots m_{n-1,n})^{\top}$$ with $m_{i,j}$'s ordered lexicographically,
and $\MM_n^\bC$  is the $(n-1)\times \binom {n}{2}$-matrix with 
columns corresponding to $m_{i,j}$. Such a column contains consecutive
elementary symmetric functions of the $(n-2)$-tuple $(-z_1,-z_2,\ldots -\hat z_i,$
$\ldots , -\hat z_j, \ldots, -z_n)$, where $\hat z_i$ and $\hat z_j$ stands for
the omission of these points. 
  
 \begin{example}
For $n=4$ the coefficients at $(1,u,u^2)$ of the numerator $P(u)$ of  \eqref{eq:main2} 
are the consecutive entries of the vector  $\MM^\bC_4\cdot \mathbf{m}_4$  where 
\begin{multline*}
\mathbf{m}_4=(m_{12},m_{13},m_{14},m_{23},m_{24},m_{34})^{\top}\qquad \text{and}\\ 
\MM^\bC_4=\begin{pmatrix} 1&1&1&1&1&1\\
-z_3-z_4&-z_2-z_4&-z_2-z_3&-z_1-z_4&-z_1-z_3&-z_1-z_2 \\
z_3z_4&z_2z_4&z_2z_3&z_1z_4&z_1z_3&z_1z_2\end{pmatrix}.
\end{multline*}
In other words, 
\[
P(u)=\sum_{\substack{1\leq i<j\leq 4\\ k<\ell,\ \{ij\}\cap\{k\ell\}=\emptyset}} 
(m_{ij}-(z_i+z_j)m_{k\ell}u+ z_i z_j m_{k\ell} u^2).
\]
\end{example}
Consider the maximal minor $Min_n^\bC$ of $\MM^\bC_n$ formed by the columns 
corresponding to $m_{12},\ldots, m_{1,n}$, i.e. the first $n-1$ columns of $\MM^\bC_n$.

 \begin{lemma}\label{lm:minor}  $det_n=(-1)^{n-1}\det(Min_n^\bC)=(-1)^{n-1}\prod_{2\le i<j\le n}(z_i-z_j).$ 

 \end{lemma}
\begin{proof}
 Indeed, the degree of $\det(Min_n^\bC)$ as a polynomial in $z_2,\dots, z_n$ equals $\binom{n-1}{2}$. We need to show that it vanishes if and only if 
$z_i=z_j$. The 'if' part is obvious since the column corresponding to $m_{1,i}$ will 
coincide with the column corresponding to $m_{1,j}$.  To see the remaining part, argue by contradiction
and assume that $(\alpha_{12},\dots,\alpha_{1n})$ is a nontrivial linear dependence among the columns of
$Min_n^\bC$. The $1k$-th column consists of the coefficients of the polynomial
$g_{1k}(u)=\frac{\prod_{j=1}^n (1-z_ju)}{1-z_ku}$, and our linear dependence is a linear
dependence among such polynomials. Evaluate these at $\frac{1}{z_j}$ and 
note that $g_{1k}(\frac{1}{z_j})$ vanish whenever $k\neq j$. Thus $\alpha_{1j}=0$, a contradiction. Thus $\det(Min_n^\bC)$ is divisible by $\prod_{2\le i<j\le n}(z_i-z_j)$.  Substituting $z_2=0, z_3=1,\dots, z_{n}=n-2$ we can check that the normalizing factor equals $(-1)^{n-1}$.  
\end{proof}

\begin{rema+}   By using Cramer's rule, 
it is not difficult to give an explicit formula for the inverse $(Min_n^\bC)^{-1}$.  
\end{rema+} 
 
From Lemma~\ref{lm:minor} we know that for any,
not necessarily non-degenerate, $S=\{z_0,z_1,\ldots, z_n\}$ 
with pairwise distinct points  the rank of $\MM_n^\bC$ equals $n-1$.  
Thus the kernel of $\MM^\bC_n$, which by definition coincides with $\M^\bC_{null}(S)$,
has dimension $\binom {n}{2}-(n-1)=\binom{n-1}{2}$.
\end{proof}

\begin{proof}[Proof of Theorem~\ref{th:more}]  The space $\M^\bR_{null}(S)\subset \M^\bC_{null}(S)$ is the maximal by inclusion real subspace of the complex kernel. In other words, it can be interpreted as  the real kernel  of the real matrix $\MM_n^\bR$  obtained by taking the real and imaginary parts of all rows of $\MM_n^\bC$. 

The case $n=2$ is trivial. The case $n=3$ can be dealt with by explicitly
computing the kernel of $\M_3^\bC$ and seeing that it does not contain real
vectors if $S$ is non-degenerate. Thus we assume $n\geq 4$.  Since the first
row of $\MM_n^\bC$ equals $(1,1,\ldots, 1)$  the matrix $\MM_n^\bR$ has size
$(2n-3)\binom{n}{2}$, see \eqref{eq:5}. Ordering $m_{ij}$'s lexicographically,
consider the maximal minor $Min_n^\bR$ of $\MM_n^\bR$ formed by the columns
corresponding to  $(2n-3)$ variables $m_{12},m_{13},\ldots , m_{1n}, m_{23},
m_{24},\ldots , m_{2n}$, i.e. the first $(2n-3)$ columns of $\MM_n^\bR$. 

\begin{lemma}\label{lm:specminor} 
$\det Min_n^\bR=C[123][124]\cdots [12n] \prod_{3\le i<j\le n}|z_i-z_j|^2,\quad 0\neq C\in\bR.$  
\end{lemma}
\begin{proof}
 We begin by showing that $\Theta:=\det Min_n^\bR$ is divisible by $[12k]$ for any $3\leq k\leq n$.
As $[12k]$ is an irreducible quadratic polynomial in $x_1, x_2, x_k$ and $y_1, y_2, y_k$, it suffices to show that vanishing of $[12k]$
implies vanishing of $\Theta$.
Vanishing of $[12k]$ is equivalent to existence of $a\in\bR$ satisfying $z_k=az_1+(1-a)z_2$.
The latter implies that $Min_n^\bR$   has linearly dependent columns ${12}$, $1k$, and $2k$. Indeed, they consist, respectively, 
of the coefficients of
{\small
\begin{align*}
g_{12}(u)=&
(1-az_1 u-(1-a)z_2 u)\times &(1-z_3u)\dots (1-z_{k-1}u)(1-z_{k+1}u)\dots (1-z_n u)\\
g_{1k}(u)=&
(1-z_2 u)\times &(1-z_3u)\dots (1-z_{k-1}u)(1-z_{k+1}u)\dots (1-z_n u)\\
g_{2k}(u)=&
(1-z_1 u)\times &(1-z_3u)\dots (1-z_{k-1}u)(1-z_{k+1}u)\dots (1-z_n u)
\end{align*}
}
which are linearly dependent:  $g_{12}(u)=(a-1)g_{1k}(u)-ag_{2k}(u)$.

To show that $\Theta$ is divisible by $|z_i-z_j|^2=(z_i-z_j)(\overline{z}_i-\overline{z}_j)$ 
for any $3\le i<j\le n$, 
observe that $z_i=z_j$ implies $g_{ki}(u)=g_{kj}(u)$ for $k=1,2$.

It remains to see that $\Theta$ is not identically $0$. Arguing by contradiction, 
let $(\alpha_{12},\alpha_{13},\dots,\alpha_{1n},\alpha_{23},\dots,\alpha_{2n})$ be the coefficients
of a nontrivial real linear dependence among the columns of $Min_n^\bR$. The latter columns correspond to
the coefficients of $g_{ij}(u)$. Evaluating these at $u=\frac{1}{z_k}$, for $3\leq k\leq n$, makes all of them
but $g_{1k}$ and $g_{2k}$ vanish. Thus 
\[
\alpha_{1k}g_{1k}({z_k^{-1}})+ \alpha_{2k}g_{2k}({z_k^{-1}})=0,\quad\text{implying }
\alpha_{1k}=-\alpha_{2k}\frac{z_k-z_2}{z_k-z_1}\quad\text{and }
 \frac{z_k-z_2}{z_k-z_1}\in\bR.
\]
A direct computation shows that the rightmost relation is equivalent to $[12k]=0$, a contradiction.
\end{proof} 
Lemma~\ref{lm:specminor} implies that for any non-degenerate $S$ the matrix $\MM_n^\bR$ has rank equal to  $2n-3$. Therefore, $\dim \M_{null}^\bR(S)=\binom 
{n}{2}-(2n-3)=\binom{n-2}{2}.$
\end{proof}

\begin{proof}[Proof of Theorem~\ref{th:5}]  For $S=\{0,z_1,z_2,z_3,z_4\}$ the space $\M_{null}^\bR(S)$ is given by the system 
\[
\MM^{\bR}_4\cdot  \mathbf{m}_4=0,\qquad\text{
where $\mathbf{m}_4=(m_{12},m_{13},m_{14},m_{23},m_{24},m_{34})^\top$ and}
\] 
\begin{equation}\label{eq:5}
\MM^{\bR}_4=\begingroup
\everymath{\scriptstyle}
\begin{pmatrix}
1 & 1 & 1 & 1 & 1 & 1 \\
-x_{3} - x_{4} & -x_{2} - x_{4} &- x_{2} - x_{3} &- x_{1} - x_{4} &- x_{1} - x_{3} &- x_{1} - x_{2} \\
-y_{3} - y_{4} & -y_{2} - y_{4} &- y_{2} - y_{3} & -y_{1} - y_{4} & -y_{1} - y_{3} &- y_{1} - y_{2} \\
x_{3} x_{4} - y_{3} y_{4} & x_{2} x_{4} - y_{2} y_{4} & x_{2} x_{3} - y_{2} y_{3} & x_{1} x_{4} - y_{1} y_{4} & x_{1} x_{3} - y_{1} y_{3} & x_{1} x_{2} - y_{1} y_{2} \\
x_{3} y_{4} + x_{4} y_{3} & x_{2} y_{4} + x_{4} y_{2} & x_{2} y_{3} + x_{3} y_{2} & x_{1} y_{4} + x_{4} y_{1} & x_{1} y_{3} + x_{3} y_{1} & x_{1} y_{2} + x_{2} y_{1}
\end{pmatrix}.
\endgroup
\end{equation}
Recall that a $k\times (k+1)$-matrix $T$ of rank $k$ has right kernel spanned by the vector
$(T^{(1)},\dots,T^{(k+1)})$, where $T^{(j)}$ is the minor of $T$ with $j$th column removed multiplied by $(-1)^{j}$.
Thus \eqref{eq:5} has a unique (up to a scaling) solution of the form: 
   $$\begin{cases}
m_{12}=&||z_1-z_2||^2[134][234]\\
m_{13}=&||z_1-z_3||^2[124][234]\\
m_{14}=&||z_1-z_4||^2[123][234]\\
m_{23}=-&||z_2-z_3||^2[124][134]\\
m_{24}=-&||z_2-z_4||^2[134][123]\\
m_{34}=-&||z_3-z_4||^2[123][124]\\
\end{cases}.$$

It is easy to prove this. We give a sketch here for $m_{12}$. Note that
$m_{12}$ equals to the determinant of the matrix $A^{(12)}$ obtained from
$\MM^\bR_4$ by removing the 1st column.  Then,
$\det A^{(12)}$ is divisible by $||z_1-z_2||^2$, as the rank of $A^{(12)}$ drops when $z_1=z_2$, and as
$||z_1-z_2||^2=(z_1-z_2)(\overline{z}_1-\overline{z}_2)$ 
is the product of two irreducible 
polynomials with complex coefficients.

Similarly, $\det A^{(12)}$ is divisible by $[234]$ (and a very similar argument applies to $[134]$).  
To see this, note that, as
$[234]$ is irreducible, it suffices to show that its vanishing implies
vanishing of $\det A^{(12)}$. To this end, assume that $z_4=az_2+(1-a)z_3$, 
with $a\in\RR$, and
make this substitution into $A^{(12)}$. The last 3 columns of  $A^{(12)}$ become 
{\small
\[ \begin{pmatrix} 
1 & 1 & 1 \\
-a x_{2} + a x_{3} - x_{1} - x_{3} & -x_{1}-x_{3} & -x_{1}-x_{2} \\
-a y_{2} + a y_{3} - y_{1} - y_{3} & -y_{1}-y_{3} & -y_{1}-y_{2} \\
a x_{1} x_{2} - a x_{1} x_{3} - a y_{1} y_{2} + a y_{1} y_{3} + x_{1} x_{3} - y_{1} y_{3} & x_{1} x_{3} - y_{1} y_{3} & x_{1} x_{2} - y_{1} y_{2} \\
a x_{1} y_{2} - a x_{1} y_{3} + a x_{2} y_{1} - a x_{3} y_{1} + x_{1} y_{3} + x_{3} y_{1} & x_{1} y_{3} + x_{3} y_{1} & x_{1} y_{2} + x_{2} y_{1}
\end{pmatrix}.
\]
}
They are linearly dependent with coefficients $(1,a-1,-a)$.
\end{proof}

\begin{proof}[Proof of Theorem~\ref{th:5-tuple}]
To prove the first part, we recall that Example~\ref{ex:gabr} settles the case
$|S|=6$. To settle the case $|S|=6+q$, we modify the latter
Example. Add $q$ points $P_1,\dots,P_q$ 
outside $\conv(T)$, so that 
so that  $P_1,\dots,P_q$ and $\sqrt{3}\pm I$ are in the convex
position, and $Q$ is the resulting convex $q+2$.  Then $F\cup Q$ and $F'\cup Q$
are equipotential $(6+q)$-gons, by additivity of the measure.

To prove the second part, we have consider the cases $|S|=3,4,5$, one by one.
Cases $|S|=3,4$ follow from  Theorem~\ref{th:more}. 

It remains to deal with the only non-trivial case $|S|=5$.  We  have to
consider the incidence matrices between the chambers and the basic simplices
for all possible non-degenerate $5$-tuples of points $S$. One can easily see
that for non-degenerate $5$-tuples there are (up to permutation of the
vertices) only $3$  different cases to consider  depending on the shape of
$\conv(S)$ which can be a $5$-gon, a $4$-gon, or a triangle. The corresponding
incidence matrices $Inc_1, Inc_2, Inc_3$ are given below 
using the labeling presented in
Fig.~\ref{fg:T3} and ~\ref{fg:T4}  for these cases.   (Greek
letters in Fig.~\ref{fg:T3} denote the vertices of the inner $5$-gon. They will
be needed below.) 
We show that in none of these case one can find a pair of equipotential
polygons. 

\begin{figure}[h]
\begin{center}
\begin{picture}(360,160)(0,0)
\put(120,20){\line(1,0){100}}
\put(220,20){\line(1,1){50}}
\put(120,20){\line(0,1){100}}
\put(120,120){\line(1,1){50}}
\put(170,170){\line(1,-1){100}}
\put(120,20){\line(3,1){150}}
\put(120,20){\line(1,3){50}}
\put(120,120){\line(1,-1){100}}
\put(120,120){\line(3,-1){150}}
\put(220,20){\line(-1,3){50}}

\put(106,11){$z_0$}
\put(210,11){$z_1$}
\put(275,70){$z_2$}
\put(155,175){$z_3$}
\put(100,115){$z_4$}

\put(170,25){$A$}
\put(203,38){$B$}
\put(193, 46){$\al$}
\put(201, 50){$\be$}
\put(190, 89){$\ga$}
\put(150, 100){$\delta$}
\put(146, 94){$\eps$}
\put(223,38){$C$}
\put(220,63){$D$}
\put(177,75){$K$}
\put(164,125){$F$}
\put(208,105){$E$}
\put(150,55){$J$}
\put(125,65){$I$}
\put(142,125){$G$}
\put(137,103){$H$}

\put(120,20){\circle*{3}}
\put(220,20){\circle*{3}}
\put(270,70){\circle*{3}}
\put(170,170){\circle*{3}}
\put(120,120){\circle*{3}}

\end{picture}
\end{center}
\caption{Chambers and their labeling  for $\conv(S)$ a $5$-gon.  \label{fg:T3}}
\end{figure}
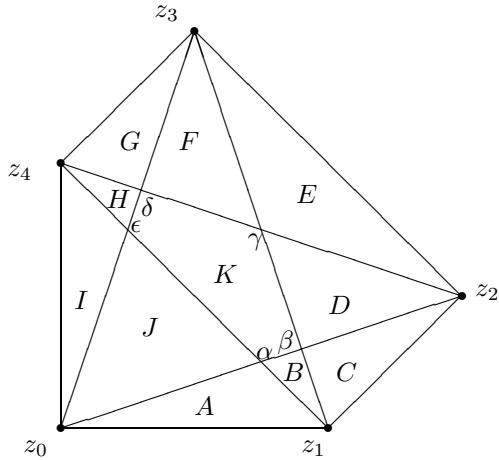

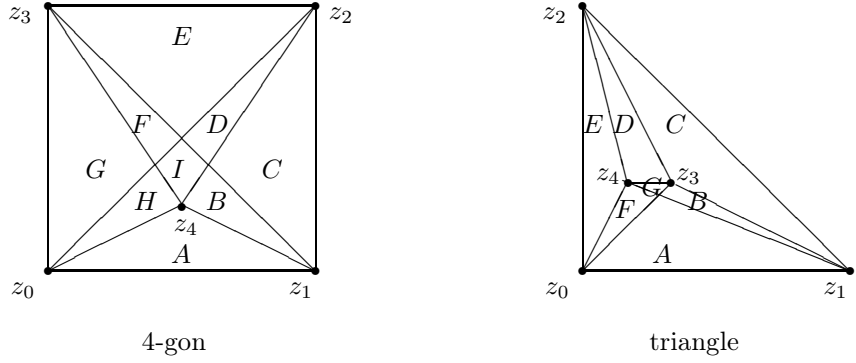
\begin{figure}[h]
\begin{center}
\begin{picture}(360,140)(0,0)
\put(40,40){\line(1,0){100}}
\put(140,40){\line(0,1){100}}
\put(40,40){\line(0,1){100}}
\put(40,40){\line(1,1){100}}
\put(40,140){\line(1,0){100}}
\put(40,140){\line(1,-1){100}}
\put(40,40){\line(2,1){50}}
\put(140,40){\line(-2,1){50}}
\put(40,140){\line(2,-3){50}}
\put(140,140){\line(-2,-3){50}}

\put(26,31){$z_0$}
\put(130,31){$z_1$}
\put(145,135){$z_2$}
\put(25,135){$z_3$}
\put(87,55){$z_4$}

\put(86,43){$A$}
\put(99,63){$B$}
\put(120,75){$C$}
\put(99,92){$D$}
\put(86,125){$E$}
\put(71,92){$F$}
\put(54,75){$G$}
\put(72,63){$H$}
\put(86,75){$I$}

\put(75,10){4-gon}

\put(40,40){\circle*{3}}
\put(140,40){\circle*{3}}
\put(140,140){\circle*{3}}
\put(40,140){\circle*{3}}
\put(90,64){\circle*{3}}

\put(240,40){\line(1,0){100}}
\put(340,40){\line(-2,1){65}}
\put(240,40){\line(0,1){100}}
\put(240,40){\line(1,1){33}}
\put(240,140){\line(1,-1){100}}
\put(240,140){\line(1,-2){33}}

\put(240,40){\line(1,2){17}}
\put(257,73){\line(1,0){15}}
\put(340,40){\line(-5,2){85}}
\put(240,140){\line(1,-4){17}}

\put(226,31){$z_0$}
\put(330,31){$z_1$}
\put(275,74){$z_3$}
\put(225,135){$z_2$}
\put(245,74){$z_4$}

\put(240,40){\circle*{3}}
\put(340,40){\circle*{3}}
\put(273,73){\circle*{3}}
\put(240,140){\circle*{3}}
\put(257,73){\circle*{3}}
  
  \put(266,43){$A$}
\put(279,63){$B$}
\put(271,92){$C$}
\put(251,92){$D$}
\put(240,92){$E$}
\put(252,60){$F$}
\put(262,68){$G$}

\put(265,10){triangle}

\end{picture}
\end{center}
\caption{Chambers and their labeling for $\conv(S)$ a $4$-gon or a triangle.  \label{fg:T4}}
\end{figure}

\begingroup
\everymath{\scriptstyle}
\scriptsize
$$Inc_1=\bordermatrix{~ & A  & B & C&D&E&F&G&H&I&J&K\cr 
\De_{012}&1&1&1&0&0&0&0&0&0&0&0\cr
\De_{013}&1&1&0&0&0&1&0&0&0&1&1\cr
\De_{014}&1&0&0&0&0&0&0&0&1&1&0\cr
\De_{023}&0&0&0&1&1&1&0&0&0&1&1\cr
\De_{024}&0&0&0&1&0&0&0&1&1&1&1\cr
\De_{034}&0&0&0&0&0&0&1&1&1&0&0}, $$

$$Inc_2={\bordermatrix{~ &A&B&C&D&E&F&G&H&I\cr 
                                       \De_{012}&1&1&1&1&0&0&0&1&1\cr
                                       \De_{013}&1&1&0&0&0&1&1&1&1\cr
                                       \De_{014}&1&0&0&0&0&0&0&0&0\cr
                                      \De_{023}&0&0&0&0&1&1&1&0&0\cr
                                      \De_{024}&0&0&0&1&0&0&0&1&1\cr
                                     \De_{034}&0&0&0&0&0&0&1&1&0}}, \
      Inc_3=\bordermatrix{~ &A&B&C&D&E&F&G\cr 
                                      \De_ {012}&1&1&1&1&1&1&1\cr
                                       \De_{013}&1&1&0&0&0&0&0\cr
                                      \De_ {014}&1&0&0&0&0&1&0\cr
                                     \De_ {023}&0&0&0&1&1&1&1\cr
                                      \De_{024}&0&0&0&0&1&0&0\cr
                                     \De_{034}&0&0&0&0&0&1&1}. $$
                                     
\endgroup

For brevity, we introduce notation $\frac{1}{2}\Ar{K}$ for the area of a polygon $K$. 
First, we need an elementary 
\begin{lemma}\label{lm:area} 
For an arbitrary triangle $\De_{\al\be\ga}$ and arbitrary secants $\al\eps$, $\be\delta$, see Fig.~\ref{fg:T5} the area of triangle $\De_{\al\be\zeta}$ is bigger than that of  $\De_{\eps\delta\zeta}$:
$$\Ar{\De_{\al\be\zeta}}>\Ar{\De_{\eps\delta\zeta}}.$$ 
\end{lemma}
\begin{figure}[h]
\begin{center}
\begin{picture}(360,170)(0,0)
\put(120,20){\line(1,0){100}}
\put(120,20){\line(1,1){75}}
\put(120,20){\line(-1,3){50}}
\put(120,20){\line(1,3){50}}
\put(220,20){\line(-1,3){50}}
\put(195,95){\line(-3,-2){65}}
\put(108,58){\line(3,-1){111}}
\put(108,59){\line(5,2){88}}

\put(106,11){$\al$}
\put(210,11){$\be$}
\put(155,175){$\ga$}
\put(124,53){$\delta$}
\put(197,95){$\eps$}
\put(143,35){$\zeta$}
\put(100, 58){$\eta$}
\put(55,175){$\kappa$}

\put(120,20){\circle*{3}}
\put(220,20){\circle*{3}}
\put(195,95){\circle*{3}}
\put(170,170){\circle*{3}}
\put(130,50){\circle*{3}}
\put(145,45){\circle*{3}}


\end{picture}
\end{center}
\caption{Illustration to Lemma~\ref{lm:area}  \label{fg:T5}}
\end{figure}
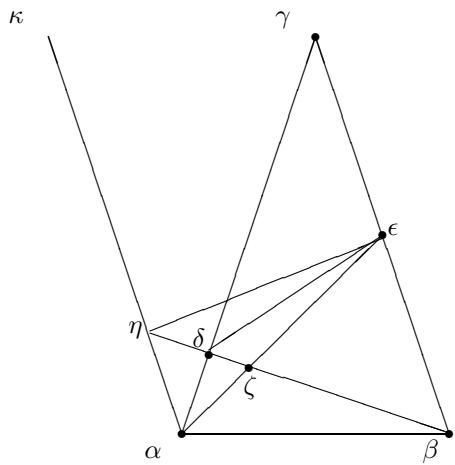
\begin{proof} Indeed, draw the line $\al\kappa$ parallel to $\beta\gamma$ and extend $\beta\delta$ till it hits $\alpha\kappa$. (The  intersection point of the latter lines is denoted by $\eta$.) Triangles $\De_{\al\be\zeta}$ and $\De_{\eta\eps\zeta}$ have equal area. Indeed, they are obtained from $\De_{\al\be\eta}$ and $\De_{\al\eps\eta}$, 
respectively, by removing $\De_{\al\zeta\eta}$.  Notice that $\De_{\al\be\eta}$ and $\De_{\al\eps\eta}$ have the same base $\al\eta$ and equal heights.
\end{proof}

\medskip

\noindent 
{\bf  Case 2.} Using labeling on the left part of Fig.~\ref{fg:T4} and \eqref{eq:dens} we conclude that densities $d_{012}, d_{014},$  $d_{023}, d_{034}$ are positive while $d_{013}, d_{024}$ are negative. From chambers $E$ and $C$ we conclude
$d_{023}=d_{012}=1$. Then from chamber $D$ we have that either $d_{024}=-1$ or $-2$. The second case leads to $d_{013}=0$,  contradiction.  Thus $d_{024}=-1$ which from chamber $I$ gives $d_{013}=-1$. Finally, $d_{034}=d_{014}=1$. Thus chambers 
$A,C,E,G$ have density $1$, chamber $I$ has density $-1$ and the remaining chambers have vanishing density. 
We need to show that $\Ar{I}<\Ar{A}+\Ar{C}+\Ar{E}+\Ar{G}$. 
We will show that actually $\Ar{I}<\Ar{C}+\Ar{G}$. 
Cut $I$ into two triangles by drawing its diagonal  connecting $z_4$ with non-neighboring vertex  
$p$ of $I$ (lying strictly above $z_4$ in the left part of Fig.~\ref{fg:T4}). Extending $z_3p$ and $z_0z_4$ we get a triangle containing $G$ and the left half of $I$ and we can apply Lemma~\ref{lm:area}. 
Analogously,  extending $z_2p$ and $z_1z_4$ we get a triangle containing $C$ and the right half of $I$ and we can apply Lemma~\ref{lm:area}. 
Thus the required measure  does not exist. 

\noindent 
{\bf  Case 3.} Using labeling on the right part of Fig.~\ref{fg:T4} and formulas~\eqref{eq:dens} we again conclude that densities $d_{012}, d_{014},$  $d_{024}$ are negative, while $d_{013}, d_{023}, d_{034}$ are positive. Similar considerations as above give
$d_{012}=d_{014}=d_{024}=-1$ and $d_{013}=d_{023}=d_{034}=1$. Thus, the densities of $A,C,E$ are $-1$ and the density of $G$ is $1$. 
In fact, $\Ar{G}<\Ar{A}$ already. Indeed, extending the interval $z_0z_4$ and $z_1z_3$ till they intersect at a point, say $p$ we get the triangle $z_0z_1p$ to which we apply Lemma~\ref{lm:area}. Thus the required measure  does not exist.

\noindent 
{\bf  Case 1.} Using labeling on Fig.~\ref{fg:T3} and \eqref{eq:dens} we see that densities $d_{012}, d_{014},$  $d_{023}, d_{034}$ are positive while $d_{013}, d_{024}$ are negative. Assuming that the densities of all chambers attain only values $0,\pm 1$ and looking at chambers $C,E,G$ we get that $d_{023}=d_{012}=d_{034}=1$. Looking at chamber $D$ we conclude that $d_{024}=-1$. (It might be equal $-2$ as well but then looking at chamber $K$ we have to conclude that $d_{013}=0$ which is impossible.) From chamber $K$ we get $d_{013}=-1$ and from chamber $I$ we get $d_{014}=1$.  Thus the density in chambers 
$A, C, E, G, I$ equals $1$, in chamber $K$ it equals $-1$ and it vanishes in the remaining chambers.  Notice that the total mass of the measure should vanish. 
To see that this cannot happen, we show that $\Ar{K}<\Ar{A}+\Ar{C}+\Ar{E}+\Ar{G}+\Ar{I}$. 
Using Lemma~\ref{lm:area} 
we conclude that $\Ar{A}>\Ar{\De_{\al\be\eps}}$,  $\Ar{C}>\Ar{\De_{\al\be\ga}}$,   
$\Ar{E}>\Ar{\De_{\be\ga\delta}}$, $\Ar{G}>\Ar{\De_{\delta\eps\ga}}$, and 
$\Ar{I}>\Ar{\De_{\eps\al\delta}}$, see Fig.~\ref{fg:T3}. 
Triangles  $\De_{\al\be\eps}$, $\De_{\al\be\ga}$, $\De_{\be\ga\delta}$, 
$\De_{\delta\eps\ga}$, $\De_{\eps\al\delta}$ pairwise overlap. 
These overlapping consists of $5$ smaller triangles inside $K$. 
The complement in $K$ to the union of triangles $\De_{\al\be\eps}$, $\De_{\al\be\ga}$, $\De_{\be\ga\delta}$, 
$\De_{\delta\eps\ga}$, $\De_{\eps\al\delta}$  is a small $5$-gon inside $K$. Now we can use these $5$ small triangles to cover the small $5$-gon inside $K$. We get exactly the same situation as the original one and  we can apply the same argument as we did and cover a substantial part of the small $5$-gon etc. Continuing this process we will in infinitely many steps exhaust the original $5$-gon $K$. 
Thus the required measure  does not exist.  
\end{proof}

To prove Theorem~~\ref{th:cone} we need the following observation. 

\begin{lemma}\label{lm:4gon}
The convex hull of the standard measures of  $4$ triangles as in Case a) Fig.~\ref{fg:T2}, i.e. two pairs forming a flip is a plane quadrangle. The convex hull of the standard measures of  $4$ triangles as in Case b) Fig.~\ref{fg:T2} is a plane triangle.
\end{lemma}

\begin{proof}
Obvious from the relations given above Fig.~\ref{fg:T2}. 
\end{proof}

\begin{proof}[Proof of Theorem~\ref{th:cone}]
Indeed if a triangle $\De$ contains an interior point other than its vertices than $\mu_\De$ is the sum of three triangles in which it is subdivided by an inner vertex, see Lemma~\ref{lm:4gon}.  (Recall that   $S$ is non-degenerate by assumption.) Thus $\mu_\De$ is not an extreme ray. On the other hand, assume that no point in $S$ other than its vertices is contained in $\De$ and $\mu_{\De}$ is a linear combination of the standard measures of some other triangles with vertices in $S$ with positive coefficients. Since no such triangle can  be contained strictly inside $\De$ by assumption and all coefficients are positive we get that any such linear combination necessarily 
has positive density somewhere outside $\De$, contradiction. \end{proof}

\section{Open problems}\label{s3}

\noindent
{\bf 1.}  Theorem~\ref{th:more} gives  the dimension of $\M_{null}^\bR(S)$ for non-degenerate $S$. Its dimension for arbitrary $S$ is unclear. On one hand, if $S$ is degenerate then $\dim \M^\bR(S)$ decreases. On the other hand, the number of equations imposed on the densities might also decrease. It seems highly plausible that $\dim \M_{null}^\bR(S)$ for an arbitrary $S$ depends only on non-oriented matroid associated to this set, see e.g. \cite{GeSe}. An algorithm calculating this dimension is given in \cite{Al2}. 

\medskip
\noindent
{\bf 2.}   Besides the cone $\mathfrak K(S)\subset \M^\bR(S)$ one can  introduce a more important, bigger, 
cone  $\mathfrak K_{pos}(S)\subset \M^\bR(S)$ where  $\mathfrak K_{pos}(S)\supset \mathfrak K(S)$ consists of all non-negative measures from $\M^\bR(S)$. 

\begin{conjecture} The combinatorial structure of $\mathfrak K_{pos}(S)$ depends only on the oriented matroid associated to $S$.
\end{conjecture}

Already for generic configurations $S$ with $6$ points  the combinatorial structure of $\mathfrak K_{pos}(S)$ and, in 
particular, the set of its extreme rays seems to be quite complicated. We plan to study  this fascinating subject in the future.   

\medskip
\noindent
{\bf 3.}  Notice that we have a natural linear map $\Psi_\mu: \M^\bR(S)\to Rat_n$ obtained by associating to each measure $\mu\in \M^\bR(S)$ its normalized generating function \eqref{eq:Psi}. Here $Rat_n$ is the linear space of rational functions of the form  $R(u)=\frac{P(u)}{\prod_{j=0}^n(1-z_ju)}, \;\; \deg P(u)\le n-2$ having real constant term. Obviously, $\dim Rat_n=2n-3$ and using Theorem~\ref{th:more} we obtain that $\M^\bR(S)$  is mapped onto $Rat_n$. 
The following question is very natural in connection with the inverse problem for the class of non-negative measures. 
 
 \begin{problem}
 Describe the extreme rays/faces of the  image cones $\Psi_\mu( \mathfrak K(S))$  and $\Psi_\mu(\mathfrak K_{pos}(S)$  in $Rat_n$.     
 \end{problem}

\medskip
\noindent
{\bf 4.}  We have an example of a pair of equipotential polygons with $|S|=6$,  
see Fig.~\ref{fg:TAD}. 

\begin{problem} Describe all $6$-tuples $S$ admitting a pair of equipotential polygons. 
\end{problem}

\end{document}